\documentclass[11pt]{amsart}

\usepackage[a4paper,hmargin=3.5cm,vmargin=4cm]{geometry}
\usepackage{amsfonts,amssymb,amscd,amstext,verbatim}

\usepackage{fancyhdr}
\usepackage[utf8]{inputenc}
\usepackage{hyperref}
\usepackage{verbatim}
\input xy
\xyoption{all}

\usepackage{times}

\usepackage{enumerate}
\usepackage{titlesec}
\usepackage{mathrsfs}

\titleformat{\section}
{\filcenter\bfseries\large} {\thesection{.}}{0.2cm}{}
\titleformat{\subsection}[runin]
{\bfseries} {\thesubsection{.}}{0.15cm}{}[.]
\titleformat{\subsubsection}[runin]
{\em}{\thesubsubsection{.}}{0.15cm}{}[.]

\usepackage[up,bf]{caption}

\newtheorem{theorem}{Theorem}[section]
\newtheorem{proposition}[theorem]{Proposition}

\newtheorem{lemma}[theorem]{Lemma}
\newtheorem{corollary}[theorem]{Corollary}

\theoremstyle{definition}
\newtheorem{definition}[theorem]{Definition}

\newtheorem{problem}[theorem]{Problem}

\numberwithin{equation}{section}
\numberwithin{figure}{section}

\newcommand\Pcal{\mathcal{P}}

\newcommand\Irm{\mathrm{I}}

\newcommand\Mrm{\mathrm{M}}

\newcommand\Orm{\mathrm{O}}

\newcommand\Urm{\mathrm{U}}

\newcommand\GL{\mathrm{GL}}
\newcommand\PGL{\mathrm{PGL}}

\newcommand\PO{\mathrm{PO}}
\newcommand\SO{\mathrm{SO}}
\newcommand\SU{\mathrm{SU}}
\newcommand\Sp{\mathrm{Sp}}

\newcommand\Ascr{\mathscr{A}}

\newcommand\Cscr{\mathscr{C}}

\newcommand\Fscr{\mathscr{F}}

\newcommand\Oscr{\mathscr{O}}

\newcommand\B{\mathbb{B}}
\newcommand\C{\mathbb{C}}
\newcommand\D{\overline{\mathbb D}}
\newcommand\CP{\mathbb{CP}}
\newcommand\HP{\mathbb{HP}}
\renewcommand\D{\mathbb D}

\renewcommand\H{\mathbb{H}}

\newcommand\N{\mathbb{N}}
\renewcommand\P{\mathbb{P}}
\newcommand\R{\mathbb{R}}

\newcommand\Z{\mathbb{Z}}

\newcommand\igot{\mathfrak{i}}
\newcommand\jgot{\mathfrak{j}}
\newcommand\kgot{\mathfrak{k}}
\renewcommand\igot{\mathfrak{i}}

\newcommand\ngot{\mathfrak{n}}

\newcommand\sgot{\mathfrak{s}}

\renewcommand\imath{\igot}

\newcommand\hra{\hookrightarrow}
\newcommand\lra{\longrightarrow}

\newcommand\wt{\widetilde}
\newcommand\wh{\widehat}
\newcommand\di{\partial}

\newcommand\dist{\mathrm{dist}}

\newcommand\Res{\mathrm{Res}}

\newcommand\cd{\overline{\D}}

\newcommand\Id{\mathrm{Id}}

\def\dist{\mathrm{dist}}

\newcommand\xistd{\xi_{\mathrm{std}}}

\numberwithin{equation}{section}

\begin{document}

\fancyhead[LO]{Proper superminimal surfaces in the hyperbolic four-space}
\fancyhead[RE]{F.\ Forstneri{\v c}} 
\fancyhead[RO,LE]{\thepage}

\thispagestyle{empty}

\vspace*{1cm}
\begin{center}
{\bf\LARGE Proper superminimal surfaces of given conformal types in the hyperbolic four-space}

\vspace*{0.5cm}

{\large\bf  Franc Forstneri{\v c}} 
\end{center}

\vspace*{1cm}

\begin{quote}
{\small
\noindent {\bf Abstract}\hspace*{0.1cm}
Let $H^4$ denote the hyperbolic four-space. Given a bordered Riemann surface, $M$, we prove that every 
smooth conformal superminimal immersion $\overline M\to H^4$ can be approximated uniformly on compacts in 
$M$ by proper conformal superminimal immersions $M\to H^4$. In particular, $H^4$ contains 
properly immersed conformal superminimal surfaces normalised by any given open Riemann surface
of finite topological type without punctures.
The proof uses the analysis of holomorphic Legendrian curves in the twistor space of $H^4$.

\vspace*{0.2cm}

\noindent{\bf Keywords}\hspace*{0.1cm}  superminimal surface, hyperbolic space, twistor space, 
complex contact manifold, holomorphic Legendrian curve 

\vspace*{0.1cm}

\noindent{\bf MSC (2020):}\hspace*{0.1cm}} 
Primary 53A10, 53C28. Secondary  32E30, 37J55.

\vspace*{0.1cm}
\noindent{\bf Date: \rm May 5, 2020}

%
%
%
%
\end{quote}

%
%

\section{Introduction}
Among minimal surfaces in an oriented four-dimensional Riemannian manifold $(X,g)$ 
there is an interesting subclass consisting of {\em superminimal surfaces of positive or negative
spin}. They were introduced in 1897 by Kommerell \cite{Kommerell1897} and were
studied by many authors; see \cite{Forstneric2020SM} for a brief historical account. 
The term {\em superminimal surface} was coined by Bryant \cite{Bryant1982JDG} (1982) in his 
seminal study of minimal surfaces in the four-sphere $S^4$ which arise as projections of holomorphic 
Legendrian curves in $\CP^3$, the Penrose twistor space of $S^4$. This {\em Bryant correspondence} 
\cite[Theorems B, B']{Bryant1982JDG} was extended to all oriented Riemannian four-manifolds $(X,g)$ 
by Friedrich \cite[Proposition 4]{Friedrich1984} who also showed in \cite{Friedrich1997} 
that superminimal surfaces in the sense of Bryant coincide with those of Kommerell. 

Assume that $M\subset X$ is a smooth oriented embedded surface with the induced conformal structure
in an oriented Riemannian four-manifold $(X,g)$. (Our considerations will also apply to immersed surfaces.) 
Then $TX|_M=TM\oplus N$ where $N=N(M)$ is the cooriented orthogonal normal 
bundle to $M$. A unit normal vector $n\in N_x$ at a point $x\in M$ determines a
{\em second fundamental form} $S_x(n):T_xM\to T_xM$, a self-adjoint linear operator
on the tangent space of $M$. For a fixed tangent vector $v\in T_x M$ we consider the closed curve 
\begin{equation}\label{eq:Ixv}
	I_x(v)=\bigl\{S_x(n)v: n\in N_x,\ |n|_g=1\bigr\} \subset T_xM.
\end{equation}

%
%
\begin{definition}\label{def:SM}
A smooth oriented embedded surface $M$ in an oriented Riemannian four-manifold
$(X,g)$ is {\em superminimal of positive (negative) spin} if for every point $x\in M$ and unit tangent vector 
$v\in T_xM$, the curve $I_x(v)\subset T_xM$ \eqref{eq:Ixv} is a circle centred at $0$ 
and the map $n \to S(n)v \in I_x(v)$ is orientation preserving (resp.\ orientation reversing). 
The last condition is void at points $x\in M$ where the circle $I_x(v)$ reduces to $0\in T_xM$.
The analogous definition applies to a smoothly immersed oriented surface $f:M\to X$. 
\end{definition}

Every superminimal surface is also a minimal surface; see Friedrich \cite[Proposition 3]{Friedrich1997}
and the discussion in \cite[Sect.\ 2]{Forstneric2020SM}. The converse only holds in special cases.
For example, every conformal minimal immersion of the two-sphere $S^2$ into the four-sphere 
$S^4$ with the spherical metric is superminimal; see \cite[Theorem C]{Bryant1982JDG} or 
\cite[Proposition 25]{Gauduchon1987A}.  The same holds for immersions of $S^2$ into the projective plane 
$\CP^2$ with the Fubini-Study metric (see \cite[Proposition 28]{Gauduchon1987A}). 
Superminimal surfaces in $S^4$ and $\CP^2$ with their natural metrics have been studied extensively; 
see the references in \cite[Sect.\ 2]{Forstneric2020SM}.

A motivation for the present paper is Bryant's theorem \cite[Corollary H]{Bryant1982JDG} 
that every compact Riemann surfaces, $M$, admits a conformal superminimal immersion into $S^4$ with 
the spherical metric. In view of the Bryant correspondence, this follows from his result 
\cite[Theorem G]{Bryant1982JDG} saying that every such $M$ admits a holomorphic Legendrian embedding 
$M\to \CP^3$ in the standard contact structure determined by the following $1$-form on $\C^4$:
\begin{equation}\label{eq:alpha}
	\alpha = z_1dz_2-z_2dz_1+z_3dz_4-z_4dz_3.
\end{equation}
Approximation theorems of Runge and Mergelyan type for Legendrian curves in $\CP^3$
have been obtained recently in 
\cite[Corollary 7.3]{AlarconForstnericLarusson2019X} and \cite[Corollary 1.11]{Forstneric2020Merg}.

In this paper we consider superminimal surfaces in the four dimensional {\em hyperbolic space} $H^4$, 
the unique simply connected complete Riemannian four-manifold of constant sectional curvature $-1$
(see \cite[Theorem 4.1]{doCarmo1992}). Among the geometric models for $H^4$ 
it will be most convenient for us to use the {\em Poincar\'e (conformal) ball model} 
given by the unit ball $\B=\{x\in \R^4:|x|^2<1\}$ endowed with the complete hyperbolic metric 
\begin{equation}\label{eq:gh}
	g_h =  \frac{4|dx|^2}{\left(1-|x|^2\right)^2},\quad\ x\in \B. 
\end{equation}
The ball model is related to the hyperboloid model in the Lorentz 
space $\R^{4,1}$ by the stereographic projection \eqref{eq:sigma}; see Sect.\ \ref{sec:H4intro}.

Recall that a {\em bordered Riemann surface} is an open domain of the form $M=R\setminus \bigcup_i \Delta_i$
in a compact Riemann surface $R$, where $\Delta_i$ are finitely many compact pairwise disjoint discs
(diffeomorphic images of $\cd=\{z\in\C:|z|\le 1\}$) with smooth boundaries $b\Delta_i$.
Its closure $\overline M$ is a {\em compact bordered Riemann surface}.

The following is our main result; it is proved in Sect.\ \ref{sec:proof} as a corollary to 
Theorem \ref{th:Omega}.

%
%
\begin{theorem}\label{th:main}
Let $M$ be a bordered Riemann surface. Every smooth conformal superminimal immersion 
$f:\overline M\to (\B,g_h)=H^4$ can be approximated uniformly on compacts in $M$ by proper 
conformal superminimal immersions $\tilde f:M\to \B$. Furthermore, $\tilde f$ can be chosen 
to agree with $f$ to a given finite order at  finitely many points in $M$.
\end{theorem}

What is new in comparison to the extant results in the literature is that we control not only the
(finite) topology of proper superminimal surfaces, but also their conformal type.

Any minimal surface in $H^4$ is open and 
its conformal universal covering is the disc (see \cite[Corollary 6.3]{Forstneric2020Merg}). 
Since every open Riemann surface with finitely generated homology group 
$H_1(M,\Z)$ is conformally equivalent to a domain obtained by removing finitely many 
closed discs and points from a compact Riemann surface (see Stout \cite[Theorem 8.1]{Stout1965TAMS}),
bordered Riemann surfaces are precisely the  open Riemann surfaces of finite topology 
without punctures. This gives the following corollary to Theorem \ref{th:main}.

%
%
\begin{corollary}\label{cor:main}
Every open Riemann surface of finite topological type without punctures is 
the conformal structure of a properly immersed superminimal surface in $H^4$.
\end{corollary}

Although Corollary \ref{cor:main} might also hold for bordered Riemann surfaces
with punctures, it is notoriously difficult to deal with this case and we leave it as
an open problem. 

It has recently been shown \cite[Corollary 6.3]{Forstneric2020Merg} that any 
self-dual or anti-self dual Einstein four-manifold (this class includes $S^4$, $H^4$, 
and many other Riemannian four-manifolds) also contains 
complete relatively compact immersed superminimal surfaces of any conformal type 
in Corollary \ref{cor:main}, thereby solving the Calabi-Yau problem for such surfaces.

Our approach to Theorem \ref{th:main} uses the Bryant correspondence
to the effect that superminimal surfaces in an oriented Riemannian four-manifold $(X,g)$ 
are the projections of horizontal holomorphic curves in total spaces of twistor bundles 
$\pi^\pm : Z^\pm\to X$, with the sign depending on the spin of the superminimal surface;
see \cite[Sect.\ 4]{Forstneric2020SM}. Both twistor spaces $Z^\pm$ of $H^4=(\B,g_h)$ 
can be identified with the domain in $\CP^3$ given by
\begin{equation}\label{eq:Omega}
	\Omega=\left\{[z_1:z_2:z_3:z_4]\in \CP^3: |z_1|^2+|z_2|^2 > |z_3|^2+|z_4|^2\right\},
\end{equation}
and the twistor projection $\pi:\Omega\to\B$ is the restriction to $\Omega$ of the twistor
projection $\pi:\CP^3\to S^4$ for the spherical metric on $S^4$ (see Sect.\ \ref{sec:H4}). 
This is a particular instance of the general fact that the twistor bundles $\pi^\pm:Z^\pm \to X$ 
of an oriented Riemannian four-manifold $(X,g)$ depend only on the conformal class of the metric $g$, 
but the horizontal bundles $\xi^\pm\subset TZ^\pm$ 
depend on the choice of a metric in that class. In the case at hand, both the spherical and the 
hyperbolic metric are conformally flat. The horizontal bundle $\xi \subset T\CP^3$ 
determined by the hyperbolic metric on $\B$ is the holomorphic contact bundle 
given by the homogeneous $1$-form
\begin{equation}\label{eq:beta}
	\beta = z_1dz_2-z_2dz_1 - z_3dz_4 + z_4dz_3
\end{equation}
(see Sect.\ \ref{sec:H4}). Compared to the $1$-form $\alpha$ \eqref{eq:alpha}, 
we note a change of sign in the last two terms. 
Although $\xi$ is contactomorphic to the standard contact structure $\xi_{\rm std}$ determined by 
$\alpha$ (in fact, $\xi_{\rm std}$ is the unique holomorphic contact structure on $\CP^3$ up to holomorphic 
contactomorphisms,  see LeBrun and Salamon \cite[Corollary 2.3]{LeBrunSalamon1994IM}),
these two structures behave very differently
with respect to the twistor projection $\pi:\CP^3\to S^4\cong \wh\R^4:=\R^4\cup\{\infty\}$. 
While $\xi_{\rm std}$ is orthogonal to all fibres of $\pi$ with respect to the Fubini-Study metric on $\CP^3$, 
$\xi$ is orthogonal to the fibres of $\pi$ over $\B$ and over the complementary open ball
$\B'=\wh\R^4\setminus \overline \B$ in the twistor metric induced by the hyperbolic metrics 
on $\B$ and $\B'$, but the fibres $\pi^{-1}(x)$ over points $x\in b\B$ are $\xi$-Legendrian curves. 
Any holomorphic Legendrian immersion $F:M\to (\CP^3,\xi)$ whose image does not lie in 
a fibre of $\pi$ determines an immersed superminimal surface in $\B$ obtained by 
intersecting the image of $F$ with $\Omega$ \eqref{eq:Omega} and projecting down to $\B$. 
If $M$ is compact and $F$ intersects $b\Omega$ transversely, 
we obtain a proper superminimal surface in $\B$ with smooth boundary in $b\B=S^3$, 
and we know by Bryant \cite[Theorem G]{Bryant1982JDG} that any compact Riemann surface embeds 
as a complex Legendrian curve in $(\CP^3,\xi)$. 
However, it seems impossible to control the conformal type of the examples obtained in this way.
In a related direction, Anderson \cite{Anderson1982} solved the Plateau problem for 
area minimizing generalized surfaces (currents) in the hyperbolic ball $\B^n$, $n\ge 3$, 
having a given boundary manifold in the sphere $b\B^n=S^{n-1}$.

On the other hand, our approach provides full control of the conformal type, but
we do not know whether the map $\tilde f:M\to H^4$ in Theorem \ref{th:main} can be chosen to extend
continuously or smoothly to the boundary of $M$. This difficulty is not unique to the present situation.
Indeed, even for the simplest minimal surfaces such as holomorphic curves
in a bounded strongly pseudoconvex domain $D$ in $\C^n$ for $n>1$ it is not known whether the analogue
of Theorem \ref{th:main} holds for maps extending smoothly to the boundary $bM$ without changing the
conformal type of $M$. (Continuous extendibility is possible in this case.)
This holds if $M$ is the disc (see Globevnik and the author \cite{ForstnericGlobevnik1992CMH}), 
or if the domain $D$ is convex and $M$ is arbitrary (see \v Cerne and Flores \cite{CerneFlores2007}). 
The most general analogue of Theorem \ref{th:main} in the holomorphic category,  
due to Drinovec Drnov\v sek and the author \cite{DrinovecForstneric2007DMJ}, pertains to
holomorphic curves in any complex manifold of dimension $>1$ having a smooth exhaustion function 
whose Levi form has at least two positive eigenvalues at every point. 
An analogue for minimal surfaces in minimally convex domains in flat Euclidean spaces $\R^n$, $n\ge 3$, 
was given by Alarc\'on et al.\ \cite[Theorems 1.1 and 1.9]{AlarconDrinovecForstnericLopez2019TAMS}.

Let us say a few words about the method of proof and the organisation of the paper.

In sections \ref{sec:H4intro} and \ref{sec:H4} we review the necessary 
background concerning the geometry of the hyperbolic space $H^4$ and its twistor space.
A more complete overview of the twistor space theory pertaining to superminimal 
surfaces is included in \cite{Forstneric2020SM}.

Our proof of Theorem \ref{th:main} relies upon the Bryant correspondence between superminimal surfaces 
in $H^4=(\B,g_h)$ and holomorphic Legendrian curves in its twistor space $(\Omega,\beta)$.
The main analytic technique used in the proof are Riemann-Hilbert modifications, using
approximates solutions of certain Riemann-Hilbert boundary value problems. One of the contributions of 
the present paper is the development of the Riemann-Hilbert modification technique for holomorphic 
Legendrian curves in projective spaces $\CP^{2n+1}$; see Theorem \ref{th:RH}. We expect that this result 
will find further applications. This classical complex-analytic method was adapted in 
\cite[Sect.\ 3]{AlarconForstnericLopez2017CM} 
to holomorphic Legendrian curves in Euclidean space $\C^{2n+1}$ with the standard contact structure 
inherited from $\CP^{2n+1}$; however, those results do not apply to 
projective spaces since the relevant geometric configurations need not be contained 
in any particular affine chart. We also prove a general position theorem showing that any noncompact
Legendrian curve in a projective space, possibly with branch points, can be approximated by holomorphic 
Legendrian embeddings; see Theorem \ref{th:emb}. 

With these newly developed tools in hand, we construct in Sect.\ \ref{sec:proof} 
properly immersed holomorphic Legendrian curves in the twistor domain $\Omega$ of $\B=H^4$
whose projections to $\B$ satisfy Theorem \ref{th:main}.
The geometry of the hyperbolic space and of its twistor space (see Secs.\ \ref{sec:H4intro}--\ref{sec:H4})
plays an essential role in the application of the Riemann-Hilbert method. 

The Riemann-Hilbert technique was used in a recent solution of the Calabi-Yau problem for 
superminimal surfaces and holomorphic Legendrian curves \cite{Forstneric2020SM}, 
and before that in the original Calabi-Yau problem concerning minimal surfaces in Euclidean spaces; 
see the formulation of the problem by S.-T.\ Yau  in \cite[p.\ 360]{Yau2000AMS} and 
\cite[p.\ 241]{Yau2000AJM}, and the recent advances summarized in
\cite{AlarconForstneric2019JAMS,AlarconForstneric2020RMI}. In the paper \cite{Forstneric2020SM} we used 
Riemann-Hilbert modifications with Legendrian discs of small extrinsic diameter, and in this case the required 
result (see \cite[Theorem 1.3]{AlarconForstneric2019IMRN}) follows from the 
Euclidean case by the contact neighbourhood theorem given by \cite[Theorem 1.1]{AlarconForstneric2019IMRN}.
On the contrary, the construction of proper Legendrian curves is more demanding since one 
must apply Riemann-Hilbert modifications with discs of big extrinsic diameter in order to push the 
boundary of the surface successively closer to the boundary of the given domain, 
thereby obtaining a proper map in the limit.

In conclusion, we mention an open problem related to Theorem \ref{th:main}.
There are constructions in the literature of infinite dimensional families of 
self-dual Einstein metrics with constant negative scalar curvature on the ball $\B\subset \R^4$  
inducing given conformal structures of a suitable type on the boundary sphere $S^3=b\B$; 
see in particular Graham and Lee \cite{GrahamLee1991}, Hitchin \cite{Hitchin1995}, and Biquard \cite{Biquard2002}. 
The twistor space of $\B$ with any such metric is a complex contact manifold. 
Does the analogue of Theorem \ref{th:main} hold true for any 
or all of these metrics, besides the standard one considered in the present paper?

%
%

\section{Making the acquaintance of the principal protagonist}\label{sec:H4intro}

In this section we recall a few basics of hyperbolic geometry that shall be used in the paper,
referring to the monograph by Ratcliffe \cite{Ratcliffe1994} for further information. 

A geometric model of the hyperbolic $n$-space $H^n$ is the hyperboloid
\begin{equation}\label{eq:Sigma}
	\Sigma=\Sigma^+ = \left\{(x_1,\ldots,x_{n+1})\in \R^{n+1}: 
	x_1^2+\cdots + x_n^2-x_{n+1}^2=-1,\ \ x_{n+1}>0 \right\}.
\end{equation}
This is one of two connected components of the unit sphere of radius $\imath=\sqrt{-1}$ in the 
Lorentz space\begin{footnote}
%
%
{Lorentz spaces are named after Hendrik Antoon Lorentz (1853--1928), a Dutch physicist and 
a 1902 Nobel Prize winner who derived the transformation equations underpinning Albert Einstein's theory 
of special relativity on the Lorentz four-space $\R^{3,1}$, the {\em Minkowski space-time}.
The terms {\em Lorentz space} and {\em Lorentz transformation} were introduced by
Poincar\'e in his 1906 paper \cite{Poincare1906}. See \cite[Sect.\ 3.6]{Ratcliffe1994} for more information.}
\end{footnote} 
%
%
$\R^{n,1}$, the Euclidean space $\R^{n+1}$ with the indefinite Lorentz inner product 
\begin{equation}\label{eq:xcircy}
	x\circ y = x_1y_1+\cdots+x_ny_n-x_{n+1}y_{n+1}.
\end{equation}
The other connected component $\Sigma^-$ is obtained by taking $x_{n+1}<0$ in \eqref{eq:Sigma}. 
Considering $x_{n+1}$ as the time variable, $\Sigma^\pm$ are contained in the 
open cone of {\em time-like vectors} 
\begin{equation}\label{eq:T}
	T=\left\{x\in \R^{n+1} : x\circ x < 0\right\}=\left\{(x',x_{n+1}) \in \R^{n+1} : x_{n+1}^2 > |x'|^2\right\},
\end{equation}	
while all the nonzero tangent vectors to $\Sigma^\pm$ are contained in the
cone $S=\{x_{n+1}^2 < |x'|^2\}$ of {\em space-like vectors}.
Their common boundary $bT=bS$ is the {\em light cone}
\[
	LC=\left\{x\in \R^{n+1} : x\circ x = 0\right\}
	=\left\{(x',x_{n+1}) \in \R^{n+1} : x_{n+1}^2 = |x'|^2\right\}.
\]
The {\em Lorentz norm} $\|x\|=\sqrt{x\circ x}$ is a positive real number for space-like vectors,
a positive multiple of $\imath=\sqrt{-1}$ for time-like vector, and it vanishes for 
{\em light-like vectors} $x\in LC$.

Consider the stereographic projection 
$\sigma:\B=\{x\in \R^n:|x|^2<1\} \stackrel{\cong}{\lra} \Sigma^+$ given by 
\begin{equation}\label{eq:sigma}
	\sigma(x) = \left(\frac{2x}{1-|x|^2},\frac{1+|x|^2}{1-|x|^2}\right), \quad\ x\in \B.
\end{equation}
The pullback by $\sigma$ of the Lorentz pseudometric $\|x\|^2 = x\circ x$ on $\R^{n,1}$
(see \eqref{eq:xcircy}) is the hyperbolic Riemannian metric of constant sectional curvature $-1$ 
on $\B$ given by \eqref{eq:gh}. The same formula defines the hyperbolic metric on the complementary ball 
\begin{equation}\label{eq:Bprime}
	\B'=\R^4\cup\{\infty\} \setminus \overline \B. 
\end{equation}
Consider the reflection $\B\to \B'$ in the sphere 
$b\B=b\B'=S^{n-1}$ given by $\B\ni x\mapsto \frac{x}{|x|^2}=y\in \B'$, with $0\mapsto \infty$. Then  
$\frac{dx}{1-|x|^2}=\frac{dy}{1-|y|^2}$, and hence the reflection is an isometry. 
The stereographic projection $\psi:\R^n\cup\{\infty\} \to S^n$ given by  
\begin{equation}\label{eq:psi}
	\psi(x) = \left(\frac{2x}{1+|x|^2},\frac{1-|x|^2}{1+|x|^2}\right),\ \ x\in \R^n;
	\quad\ \psi(\infty) = \sgot = (0',-1)
\end{equation}
maps the balls $\B,\B'$ onto opposite hemispheres of the Euclidean sphere $S^n\subset \R^{n+1}$. 

The group of linear automorphisms of $\R^{n,1}$ preserving the Lorentz inner product $x\circ y$
is the {\em Lorentz group} $\Orm(n,1)$. Every Lorentz transformation preserves the 
light cone $LC$ and the open cones $T,S$ of time-like and space-like vectors, but it may interchange 
the two connected components $T^\pm$ of $T$ \eqref{eq:T} defined by $\pm\, x_{n+1}>0$. 
The group $\Orm(n,1)$ contains the index two {\em positive Lorentz group} $\PO(n,1)$ of 
Lorentz transformations mapping $T^+$ (and hence also $T^-$) to itself.
Since $\Sigma$ \eqref{eq:Sigma} is the component of the unit sphere of radius $\imath=\sqrt{-1}$ 
contained in $T^+$, the restriction of any positive Lorentz transformation $A\in \PO(n,1)$ 
to $\Sigma\cong H^n$ is an isometry of $H^n$; conversely, every isometry of $H^n$
extends to a unique $A\in \PO(n,1)$ (see \cite[Theorem 3.2.3]{Ratcliffe1994}).

Via the stereographic projection $\sigma:\B\to\Sigma$ given by \eqref{eq:sigma}, the group $\PO(n,1)$
of positive Lorentz transformations of $\R^{n,1}$ corresponds to the group $\Irm(\B)$ of isometries of the 
hyperbolic $n$-ball $(\B,g_h)$ \eqref{eq:gh}. Note that $\Irm(\B)$ coincides with the group 
$\Mrm(\B)$ of M\"obius transformations of the extended Euclidean space 
$\wh\R^n=\R^n\cup\{\infty\}$ mapping $\B$ 
onto itself (see Ratcliffe \cite[Theorem 4.5.2 and Corollary 1]{Ratcliffe1994}). Every M\"obius transformation
in $\Mrm(\B)$ is a composition of reflections of $\wh\R^n$ in spheres orthogonal
to $S^{n-1}=b\B$, where spheres passing through $\infty$ are hyperplanes through 
$0\in\R^n$. In particular, $\Mrm(\B)=\Mrm(\B')$ where $\B'$ is the complementary 
hyperbolic ball \eqref{eq:Bprime}. The restriction of the elements of $\Mrm(\B)$ to the sphere 
$S^{n-1}=b\B$ is the M\"obius  group $\Mrm(S^{n-1})$.

An important class of objects in $H^n$ are its hyperbolic planes. 
A vector subspace $V\subset \R^{n,1}$ is said to be {\em time-like} if it contains a 
time-like vector, i.e., $V$ intersects the cone $T$ \eqref{eq:T}. 
A {\em hyperbolic $m$-plane} of $H^n$ is the intersection of $H^n=\Sigma$ \eqref{eq:Sigma} 
with an $(m+1)$-dimensional time-like vector subspace of $\R^{n,1}$. 
Hyperbolic lines are precisely the geodesics of $H^n$ (cf.\ \cite[p.\ 70]{Ratcliffe1994}). 
Preimages of hyperbolic $m$-planes in $H^n$ by the stereographic projection $\sigma:\B\to \Sigma=H^n$
\eqref{eq:sigma} are called hyperbolic $m$-planes in $\B$; every such is the intersection
of $\B$ with either an $m$-dimensional vector subspace of $\R^n$ or an $m$-sphere orthogonal
to the unit sphere $S^{n-1}=b\B$ \cite[Theorem 4.5.3]{Ratcliffe1994}. 
Every hyperbolic $m$-plane $\Lambda\subset \B$ is the image of
the $m$-ball $\B\cap V\cong \B^m$ in an $m$-dimensional vector subspace $V\subset\R^n$
by the orientation preserving {\em hyperbolic translation} $\tau_b\in \Mrm(\B)$ for some $b\in \B$: 
\begin{equation}\label{eq:taub}
	\tau_b(x) = \frac{1}{|b|^2|x|^2+2x\,\cdotp b+1} 
	\left((1-|b|^2)x+(|x|^2+2x\,\cdotp b+1)b\right).
\end{equation}
(See \cite[(4.5.5)]{Ratcliffe1994}. Here, $x\,\cdotp b$ denotes the Euclidean inner product on $\R^n$. 
Note that $\tau_b(0)=b$ and $\tau_0=\Id$.)
Indeed, choosing $b\in \Lambda$ with the smallest Euclidean norm $|b|$ and 
letting $V=T_b\Lambda$ considered as a vector subspace of $\R^n$, we have that 
$\tau_b(\B\cap V)=\Lambda$. To see this, note that for every $x\in \B\cap V$ and $h\in V$ 
we have that $x\,\cdotp b=0$ and hence
\[ 
	\tau_b(x) =  \frac{1}{|b|^2|x|^2+1} \left((1-|b|^2)x+(|x|^2+1)b\right),
	\quad (d\tau_b)_0 h = (1-|b|^2) h.
\]
Since a hyperbolic plane is uniquely determined by a pair $(b,V)$ where $V$ is an $m$-dimensional 
vector subspace of $\R^n$ and $b\in \B$ is orthogonal to $V$, the claim follows. 
We summarise this observation for a later application.

%
%
\begin{proposition}\label{prop:Lambda}
For each pair $(b,V)$, where $V$ is an $m$-dimensional vector subspace of $\R^n$ 
and $b\in \B$ is orthogonal to $V$, there is a unique hyperbolic $m$-plane
$\Lambda(b,V)\subset  \B$ with
\[
	b\in \Lambda(b,V), \quad T_b \Lambda(b,V)=V, \quad |b|=\min\{|x|:x\in \Lambda(b,V)\}.
\]  
We have that $\Lambda(0,V)=\B\cap V$, and if $b\ne 0$ then 
\begin{equation}\label{eq:LbV}
	\Lambda(b,V) = \tau_b(\B\cap V)= \B\cap S^m(a,r), 
\end{equation}
where $a\in \R_+ \cdotp b$ is the unique point with $|a|=\frac{1+|b|^2}{2|b|}$,
$r=\frac{1-|b|^2}{2|b|}$, and $S^m(a,r)$ is the sphere with centre $a$ and radius $r$
in the $(m+1)$-dimensional vector subspace $L\subset \R^n$ spanned by $V$ and $b$.
In particular, $\Lambda(b,V)$ depends real analytically on $(b,V)$.
\end{proposition}

In the calculation of the centre point $a$ and the radius $r$ of the sphere 
$S^m(a,r)\subset L\cong \R^{m+1}$ in the above proposition, one takes into account that 
$S^{m}(a,r)$ intersects the unit sphere $S^m\subset \R^{m+1}$ orthogonally if and only if
$|a|^2=r^2+1$  \cite[Theorem 4.4.2]{Ratcliffe1994}.

%
%

\section{Twistor space of the hyperbolic four-space}\label{sec:H4}

We briefly recall the main facts about twistor spaces pertaining to this paper, 
referring to \cite[Sect.\ 4]{Forstneric2020SM} and the references therein for a more complete account. 

To any  smooth orientable Riemannian four-manifold $(X,g)$ one associates
a pair of twistor fibre bundles $\pi^\pm : Z^\pm \to X$. 
The fibre $\pi^{-1}(x)\cong \CP^1$ over any point $x\in X$ consists of almost
hermitian structures $J:T_xX\to T_xX$, that is, linear maps satisfying $J^2=-\Id$ preserving
the metric $g$ and either agreeing or disagreeing with the orientation of $X$, depending on $\pm$.
The Levi-Civita  connection of $(X,g)$ determines a horizontal subbundle 
$\xi^\pm\subset TZ^\pm$ projecting by $d\pi^\pm$ isomorphically onto the 
tangent bundle of $X$. The total spaces $Z^\pm$ carry almost complex 
structures $J^\pm$ such that $\xi^\pm$ is a $J^\pm$-complex subbundle, for any point $z\in Z^\pm$
(an almost hermitian structure on $T_xX$ for $x=\pi^\pm(z)$)
we have that $d\pi^\pm_z \circ J^\pm_z = z\circ d\pi^\pm_z $, and $J^\pm$ coincides with 
the natural complex structure on the fibres $(\pi^\pm)^{-1}(x)\cong \CP^1$.
The bundles $(Z^\pm,J^\pm)$ only depend on the conformal class of $g$, 
but the horizontal bundle $\xi^\pm$ depends on the choice of $g$ in a given conformal class.

Let $M$ be an open Riemann surface. The {\em Bryant correspondence} says that conformal 
superminimal immersions $f:M\to X$ of positive or negative spin (see Def.\ \ref{def:SM})
are the twistor projections of horizontal (tangent to $\xi^\pm$) 
holomorphic immersions $F^\pm:M\to Z^\pm$, the sign depending on the spin of $f$
(see Bryant \cite[Theorems B, B']{Bryant1982JDG}, Friedrich \cite[Proposition 4]{Friedrich1984},
and the summary in \cite[Theorem 4.6]{Forstneric2020SM}). 
According to Atiyah et al.\ \cite[Theorem 4.1]{AtiyahHitchinSinger1978}, 
the almost complex structure $J^\pm$ is integrable (i.e., $(Z^\pm,J^\pm)$ is a 
complex manifold) if and only if the Weyl tensor $W=W^++W^-$ 
(the conformally invariant part of the curvature tensor of $g$)
satisfies $W^+=0$ ($g$ is anti-self-dual) or $W^-=0$ ($g$ is self-dual), respectively.
If either of these conditions hold then the corresponding 
horizontal subbundle $\xi^\pm\subset TZ^\pm$ 
is a holomorphic subbundle if and only if $g$ is an Einstein metric, and in such case $\xi^\pm$ is 
a holomorphic contact bundle if and only if $g$ has nonzero (constant) scalar curvature;
see Salamon \cite[Theorem 10.1]{Salamon1983} and Eells and Salamon \cite[Theorem 4.2]{EellsSalamon1985}.

The spherical metric on $S^4$ and the hyperbolic metric on $H^4$ are conformally flat
Einstein metrics of curvature $\pm1$, so their twistor spaces are complex contact three-manifolds. 
It was shown by Penrose \cite[Sect.\ VI]{Penrose1967} and 
Bryant \cite[Sect.\ 1]{Bryant1982JDG} that both twistor spaces $Z^\pm(S^4)$ 
can be identified with the complex projective space $\CP^3$ with the Fubini-Study metric 
such that the horizontal bundle is the holomorphic contact bundle $\xistd\subset T\CP^3$ given by 
$\alpha$ \eqref{eq:alpha}. An elementary proof is given in \cite[Sect.\ 6]{Forstneric2020SM}.
It is also known (see Friedrich \cite{Friedrich1997}) that the twistor spaces 
$Z^\pm$ of the hyperbolic space $H^4=(\B,g_h)$ \eqref{eq:gh} 
can be identified with the domain $\Omega\subset\CP^3$ \eqref{eq:Omega} 
with the contact structure $\xi$ defined by $\beta$ \eqref{eq:beta}. Since we shall need a 
more precise understanding of the relevant geometry, 
we recall the main facts. 

Let $\H$ denote the field of quaternions. An element of $\H$ is written uniquely as
\[ 
	q= x_1+x_2\igot  + x_3\jgot  + x_4 \kgot  = z_1+z_2\jgot,
\] 
where $(x_1,x_2,x_3,x_4)\in\R^4$, $z_1=x_1+x_2 \igot \in\C,\ z_2=x_3+x_4\igot \in \C$, and  
$\igot,\jgot,\kgot$ are the quaternionic units. We identify $\R^4$ with $\H$ using $1,\igot,\jgot,\kgot$ 
as the standard basis. Recall that 
\[
	\bar q=x_1-x_2\igot  - x_3\jgot  - x_4 \kgot,\quad 
	q\bar q=|q|^2=\sum_{i=1}^4 x_i^2, \quad 
	q^{-1}=\frac{\bar q}{|q|^2}\ \text{if}\ q \ne 0,\quad \overline{pq}=\bar q \bar p. 
\]
We identify the quaternionic plane $\H^2$ with $\C^4$ by 
\begin{equation}\label{eq:H2C4}
	\H^2\ni q=(q_1,q_2)=(z_1+z_2\jgot,z_3+z_4\jgot) = (z_1,z_2,z_3,z_4)=z \in \C^4.
\end{equation}
Write $\H^2_*=\H^2 \setminus \{0\}\cong \C^4_*$. The situation is described by the following diagram
\[ 
	\xymatrixcolsep{3pc}
	\xymatrix{
	\H^2_* \ar[r]^{\phi_1} \ar[dr]^{\phi} & \CP^3 \ar[d]_{\pi}  
	\\ 
	\wh\R^4 \ar@{<->}[r]^{\cong} & \H\P^1 \ar[r]^{\psi} & S^4
	}
\]
where $\wh\R^4=\R^4\cup\{\infty\}$,  
$\phi_1:\H^2_*\cong \C^4_*\to \CP^3$ is the canonical projection with fibre $\C^*$
sending $q\in \H^2_*$ to the complex line $\C q\in \CP^3$, $\pi:\CP^3\to \HP^1$ is the fibre bundle 
sending a complex line $\C q$ $(q\in \H^2_*)$ to the quaternionic line $\H q=\C q \oplus \C \jgot q$,
and $\phi=\pi\circ\phi_1$ sends $q\in \H^2_*$ to $\H q\in \HP^1$. The fibre $\pi^{-1}(\pi(\C q))$ is the 
linear rational curve $\CP^1\subset \CP^3$ of complex lines contained in the quaternionic line
$\H q$. Thus, $\HP^1$ is the one-dimensional quaternionic projective space which we identify with 
$\H\cup\{\infty\} = \wh\R^4$ such that the quaternionic line $\{0\}\times \H$ corresponds to $\infty$. 
The map $\psi:\wh\R^4\stackrel{\cong}{\lra} S^4$ is the stereographic projection \eqref{eq:psi}.
With these identifications, the projection $\pi:\CP^3\to \HP^1$ is the twistor bundle $Z^+(S^4)\to S^4$. 
We get the negative twistor bundle $Z^-(S^4)\to S^4$ by reversing the orientation on $S^4$; 
for example, by replacing the stereographic projection $\psi$ by the one
sending $\infty$ to $\ngot=(0,0,0,0,1)\in S^4\subset \R^5$. Using the coordinates \eqref{eq:H2C4} we have
\[
	\phi(q_1,q_2) = q_1^{-1}q_2=\frac{1}{|q_1|^2}\bar q_1 q_2 
	= \frac{1}{|z_1|^2+|z_2|^2}\left(\bar z_1 z_3+z_2\bar z_4,\bar z_1z_4-z_2\bar z_3\right).
\]
Identifying $\HP^1\cong \R^4\cup \{\infty\}=\C^2\cup \{\infty\}=:\wh \C^2$ and using complex coordinates 
$w=(w_1,w_2)\in\C^2$, this shows that the twistor projection $\pi:\CP^3\to \wh \C^2$ is given by
\begin{equation}\label{eq:pi}
	w_1=\frac{\bar z_1 z_3+z_2\bar z_4}{|z_1|^2+|z_2|^2},\quad\ 
	w_2=\frac{\bar z_1z_4-z_2\bar z_3}{|z_1|^2+|z_2|^2},	
	\quad\ |w|^2=\frac{|z_3|^2+|z_4|^2}{|z_1|^2+|z_2|^2}= \frac{|q_2|^2}{|q_1|^2}.
\end{equation}
The maximal subgroup $G_s\subset \GL_4(\C)$ 
which passes down to the group of biholomorphic isometries of $\CP^3$ in the Fubini-Study metric, 
and further down to the group of isometries of $\HP^1\cong \wh\R^4$ in the spherical metric 
$g_s = \frac{4|x|^2}{\left(1+|x|^2\right)^2}$ $(x\in\R^4)$, is the group preserving 
the quaternionic inner product on $\H^2$ given by 
\begin{equation}\label{eq:IP1}
	\H^2\times \H^2\ni (p,q)\ \longmapsto\ p \bar q^t = p_1\bar q_1+p_2\bar q_2 \in\H.
\end{equation}
(We consider elements of $\H^2$ as row vectors acted upon by right multiplication.) Writing
\begin{equation}\label{eq:pq}
	p=(z_1+z_2\jgot,z_3+z_4\jgot)=z,\quad  q=(v_1+v_2\jgot,v_3+v_4\jgot)=v, 
\end{equation}
a calculation gives
\begin{equation}\label{eq:pbarq}
	p \bar q^t = z \, \overline v^t + \alpha_0(z,v) \jgot,\quad\ \alpha_0(z,v)=z_2v_1-z_1v_2+z_4v_3-z_3v_4. 
\end{equation}
Then $\alpha_0(z,dz)=\alpha$ is the contact form \eqref{eq:alpha}. 
Denoting by $J_0\in \SU(4)$ the matrix having diagonal blocks 
$\left(\begin{smallmatrix}0&-1\\1&0\end{smallmatrix}\right)$ and zero off-diagonal blocks, 
we have $\alpha_0(z,v)=z J_0 v^t$ and hence
\begin{equation}\label{eq:Gs}
	G_s = \{A\in \Urm(4): AJ_0A^t=J_0\} = \Urm(4) \cap\Sp_2(\C),
\end{equation}
where $\Sp_2(\C)$ denotes the complexified symplectic group.

%
%
We now consider the twistor space $Z^+$ of the hyperbolic space $H^4=(\B,g_h)$.
From \eqref{eq:pi} we see that the preimage of $\B$ by the twistor projection 
$\pi:\CP^3\to \wh\C^2$ is the domain 
\begin{equation}\label{eq:Omega2}
	\Omega=\pi^{-1}(\B) = \left\{[z_1:z_2:z_3:z_4]\in \CP^3: |z_1|^2+|z_2|^2 > |z_3|^2+|z_4|^2\right\}.
\end{equation}
Likewise, the preimage $\Omega'=\pi^{-1}(\B')\subset \CP^3$ obtained
by reversing the inequality in \eqref{eq:Omega2} is the twistor space of the complementary
four-ball $\B'$ \eqref{eq:Bprime} with the hyperbolic metric. The common boundary of these 
two domains is the cone
\begin{equation}\label{eq:Gamma}
	\Gamma=\left\{[z_1:z_2:z_3:z_4]\in \CP^3: |z_1|^2+|z_2|^2 = |z_3|^2+|z_4|^2\right\}
\end{equation}
whose projection $\pi(\Gamma)$ is the unit sphere $b\B=b\B'=S^3\subset \R^4$. 
The maximal subgroup $G_h$ of $\GL_4(\C)$ descending to a group of holomorphic 
automorphisms of $\CP^3$, and further down to the group of isometries 
$\Irm(\B)=\Irm(\B')\subset \Mrm(\wh \R^4)$ of the hyperbolic balls $\B$ and $\B'$, 
consists of all $A\in \GL_4(\C)$ preserving the indefinite quaternionic inner product
\[ 
	\H^2\times \H^2\ni (p,q)\ \longmapsto\ p* q = p_1\bar q_1-p_2\bar q_2 \in\H.
\] 
Writing $(p,q)$ in the complex notation \eqref{eq:pq} we have that 
\begin{equation}\label{eq:quadratic2}
	p * q = \left(z_1\bar v_1 + z_2 \bar v_2 - z_3\bar v_3-z_4\bar v_4\right) + 
	\left(z_2v_1-z_1v_2 - z_4v_3 + z_3v_4\right) \jgot. 
\end{equation}
The subgroup of $\GL_4(\C)$ preserving the first component on the right hand side is
$\Urm(2,2)$. Let $\beta_0(z,v)$ denote the coefficient of $\jgot$ in \eqref{eq:quadratic2}. 
Note that  $\beta_0(z,dz)=\beta$ is the form \eqref{eq:beta}.
Let $J_1\in \SU(4)$ be the matrix having the diagonal blocks 
$\left(\begin{smallmatrix}0&-1\\1&0\end{smallmatrix}\right)$,  
$\left(\begin{smallmatrix}0&1\\-1&0\end{smallmatrix}\right)$ and zero off-diagonal blocks.
Then $\beta_0(z,v)=z J_1 v^t$, so the group we are looking for is
\begin{equation}\label{eq:Gh}
	G_h = \{A\in \Urm(2,2)  : AJ_1A^t=J_1\}.
\end{equation}
This also shows that the horizontal bundle of the twistor projections 
$\pi:\Omega\to \B$ and $\pi:\Omega'\to \B'$ is the kernel $\xi\subset T\CP^3$ of 
the $1$-form $\beta$ \eqref{eq:beta}, a contact bundle. 

For any $p=(p_1,p_2)\in \H^2_*$ the fibre $\pi^{-1}(\phi(p))\subset\CP^3$ is the space of all
complex lines contained in the quaternionic line $\H p$. 
The tangent space to this fibre at any point is spanned by a vector $q=a p$ for some 
imaginary quaternion $a \in \H$ with $|a|=1$. From \eqref{eq:quadratic2} we get
\[
	p* q = p_1\overline{ap_1} - p_2\overline{ap_2}
		     = p_1\bar p_1 \bar a - p_2\bar p_2 \bar a
		     = (p* p)\overline a.
\]
This vanishes for all $a\in \H$ precisely when $|p_1|^2=|p_2|^2$ which is equivalent to 
$\phi_1(p)\in \Gamma=\pi^{-1}(b\B)$ \eqref{eq:Gamma}. It follows that for every point $x\in b\B=b\B'$ 
the fibre $\pi^{-1}(x)\subset \Gamma$ is a $\xi$-Legendrian curve.
This is in strong contrast to the situation for the contact bundle $\xi_{\rm std}$ which is transverse 
to all fibres of $\pi$. This difference reflects the fact that the hyperbolic metrics on $\B$ and $\B'$ 
blow up along their common boundary sphere. 

The above discussion in illustrated by the following diagram, where $G_h$ is the
group \eqref{eq:Gh} and $\Mrm(\B)$ is the M\"obius group (the isometry group)
of $\B$ introduced in Sect.\ \ref{sec:H4intro}.
\[ 
	\xymatrixcolsep{4pc}
	\xymatrix{
	& \H^2_* \ar[r]^{A\,\in\, G_h} \ar[d]_{\phi_1}  & \H^2_* \ar[d]^{\phi_1} & 
	\\
	\Omega\, \ar@{^{(}->}[r] \ar[d]_{\pi} & \CP^3  \ar[r]^{\wt A \,\in\, \PGL(4)} \ar[d]_{\pi} 
	& \CP^3  \ar[d]^{\pi}   &   \ \Omega \ar@{_{(}->}[l] \ \ar[d]^{\pi} 
	\\
	\B \, \ar@{^{(}->}[r] & \wh\R^4  \ar[r]^{A'\,\in\, \Mrm(\B)}  	
	& \wh\R^4 &  \ \B \ \ar@{_{(}->}[l] 
	}
\]

The negative twistor bundle $Z^-(\B)$ is the positive twistor bundle of $\B$ with the opposite 
orientation; it can still be identified with $(\Omega,\beta)$. There is however no need to consider it 
since an orientation reversing isometry $\tau:\B\to \B$ (for example, a reflection in a hyperplane
of $\R^4$ through the origin) maps any conformal superminimal surface $f:M\to \B$ of negative spin to 
a conformal superminimal surface $\tau\circ f: M\to \B$ of positive spin and vice versa; 
hence it suffices to consider superminimal surfaces of positive spin.

%
%
\vspace{2mm}
\begin{proposition}\label{prop:Lz}
Every oriented hyperbolic $2$-plane $\Lambda(b,V)\subset  \B=\B^4$ in Proposition \ref{prop:Lambda}  
is a totally geodesic superminimal surface in $(\B,g_h)$, hence a superminimal surface of both 
positive and negative spin. Its twistor lift to the domain $\Omega\subset \CP^3$ \eqref{eq:Omega2} 
is the intersection of $\Omega$ with a linear $\xi$-Legendrian rational curve $\CP^1\subset \CP^3$.
\end{proposition}

\begin{proof}
For any two-plane $0\in V\subset \R^4$, $\Lambda(0,V)=\B\cap V$ is a hyperbolic disc in the metric $g_h$. 
Given a circle $C\subset V$ intersecting $b\B\cap V$ orthogonally in $V$, $C$ also intersects
$b\B$ orthogonally in $\R^4$, so $C\cap \B$ is a geodesic of $(\B,g_h)$.
This shows that $\B\cap V$ is a totally geodesic surface in $\B$, hence superminimal
with all circles $I_x(v)$ in Def.\ \ref{def:SM} reducing to points. (In particular, $\B\cap V$ with 
any orientation is superminimal of both $\pm$ spin.) 
Taking $V=\R^2\times \{0\}^2=\C\times \{0\}$ and the parameterization 
$f(\zeta)=(\zeta,0)\in \B\cap V$ for $\zeta\in \D=\{|\zeta|<1\}$, we see from  \eqref{eq:pi} that the 
holomorphic $\xi$-Legendrian embedding
\[
	F:\CP^1=\C\cup\{\infty\} \hra \CP^3,\quad\  F(\zeta) = [1:0:\zeta:0] 
\]
restricted to the disc $\D$ is the twistor lift of $f$. (Note that $F$ is also 
$\xistd$-Legendrian, so this particular map $f$ is also a superminimal surface in $S^4$
with the spherical metric.) Reversing the orientation on $V=\R^2\times \{0\}^2$, a conformal 
orientation preserving parameterization of $\B\cap V$ is $f(\zeta)=(\bar\zeta,0)$ $(\zeta\in\D)$ 
with the twistor lift $F(\zeta) = [0:1:0:\zeta]\in\Omega$.

Any other hyperbolic surface $\Lambda(b,V)$ can be obtained from $\B\cap (\R^2\times \{0\}^2)$ by an 
orientation preserving isometry of $\B$. Indeed, we get other planes through $0$ by orthogonal rotations 
in $\SO(4)$, and for $0\ne b\in \B$ we have that $\Lambda(b,V) = \tau_b(\B\cap V)$ (cf.\ \eqref{eq:LbV})
where $\tau_b\in\Mrm(\B)$ is the orientation preserving hyperbolic translation \eqref{eq:taub}.	
Since every orientation preserving isometry of $\B$ lifts to a holomorphic contactomorphism 
of $(\CP^3,\xi)$ preserving the domain $\Omega$, the same conclusion holds for the twistor lift of every surface 
$\Lambda(b,V)$.
\end{proof}

%
%

\section{The Riemann-Hilbert method for Legendrian curves} 
\label{sec:RH}

In this section we develop the Riemann-Hilbert deformation method for holomorphic 
Legendrian curves in complex projective spaces; see Theorem \ref{th:RH}.

The Riemann-Hilbert deformation method for holomorphic curves and related geometric 
objects is a very useful tool in global constructions of such object having certain additional 
properties. A particularly useful feature of this technique is that it offers a 
precise geometric control of the placement of the object into the ambient space; 
this is especially helpful if one aims to preserve its conformal (complex) structure without 
having to cut away pieces of it during an inductive construction. 
It is therefore not surprising that this technique has been used in constructions 
of proper holomorphic maps from bordered Riemann surfaces into an optimal class of 
complex manifolds and complex spaces (see \cite{DrinovecForstneric2007DMJ} and the references therein), 
of complete holomorphic curves which are either proper in a given domain or
contained in a small neighbourhood of a given curve (see \cite{AlarconForstneric2013MA}), 
in the Poletsky theory of disc functionals (see \cite{DrinovecForstneric2012IUMJ}), and others. 
In recent years this method has been adapted to several other geometries, 
in particular to conformal minimal surfaces in Euclidean spaces $\R^n$ and holomorphic 
null curves in $\C^n$ for any $n\ge 3$ (see the survey \cite{AlarconForstneric2019JAMS}) and 
to holomorphic Legendrian curves in $\C^{2n+1}$ with its standard complex contact structure 
(see \cite{AlarconForstnericLopez2017CM}). 

The following is the main result of this section. 
Since the contact structure on $\CP^{2n+1}$ is unique up to holomorphic contactomorphisms 
(see \cite[Corollary 2.3]{LeBrunSalamon1994IM}), the precise choice of the contact bundle is irrelevant.

%
%
\begin{theorem}[The Riemann-Hilbert method for Legendrian curves in $\CP^{2n+1}$]\label{th:RH}
Assume that $M$ is a compact bordered Riemann surface, 
$I\subset bM$ is an arc which is not a boundary component of $M$,
$f: M \to\CP^{2n+1}$ is a Legendrian map of class $\Ascr^1(M)=\Cscr^1(M)\cap \Oscr(\mathring M)$, 
and for every $u\in bM$ the map $\cd \ni v\mapsto F(u,v) \in \CP^{2n+1}$
is a Legendrian disc of class $\Ascr^1(\cd)$ depending continuously on $u\in bM$ such that 
$F(u,0)=f(u)$ for all $u\in bM$ and $F(u,v)=f(u)$ for all $u\in bM\setminus I$ and $v\in\cd$. 
Assume that there is a projective hyperplane $H \subset \CP^{2n+1}$ which avoids 
the compact set $\bigcup_{u\in I}F(u,\cd)$. 
Given a number $\epsilon>0$ and a neighbourhood $U\subset M$ of the arc $I$, there exist 
a holomorphic Legendrian map $\tilde f : M \to\CP^{2n+1}$ and a neighbourhood $V\Subset U$ of $I$ 
with a smooth retraction $\tau : V\to V\cap bM$ such that the following conditions hold.
\begin{enumerate}[\rm (i)] 
\item  $\dist(\tilde f(u),f(u)) <\epsilon$ for all $u\in M\setminus V$.
\vspace{1mm}
\item  $\dist(\tilde f(u),F(u,b\D))<\epsilon$ for all $u \in bM$.
\vspace{1mm}
\item  $\dist(\tilde f(u),F(\tau(u),\overline{\D}))<\epsilon$ for all $u\in V$.
\vspace{1mm}
\item $\tilde f$ agrees with $f$ to a given finite order on a given finite set of points in $\mathring M$. 
\end{enumerate}
\end{theorem}

Recall that a map from a compact bordered
Riemann surface $M$ to a complex manifold $X$ is called holomorphic if it extends to a holomorphic map
$U\to X$ from an open neighbourhood of $M$ in an ambient Riemann surface $R$.

\begin{proof}
We adapt the proof of \cite[Theorem 3.3]{AlarconForstnericLopez2017CM} (which applies to 
holomorphic Legendrian curves in $\C^{2n+1}$) to the projective case.
For simplicity of notation we consider the case $n=1$; the same proof applies in general
with obvious modifications.

We begin with a few reductions. We may assume that $M$ is connected,
$f$ is nonconstant, and its image $f(M)$ is not contained in the affine chart $\CP^3\setminus H$,
for otherwise the result follows from  \cite[Theorem 3.3]{AlarconForstnericLopez2017CM}.
The special case when $M=\cd$ and the entire configuration is contained in an 
affine chart $\C^3\subset\CP^3$ is furnished by \cite[Lemma 3.2]{AlarconForstnericLopez2017CM}.

By Bertini's theorem (see \cite[p.\ 150]{GoreskyMacPherson1988} or \cite{Kleiman1974}
and note that this is an application of the transversality theorem) 
we can move the hyperplane $H$ slightly to ensure that it intersects $f(M)$ transversely 
and it does not meet the compact set $f(bM)\cup \bigcup_{p\in I}F(p,b\D)$.  
In particular, $f$ is an immersion at any point $p\in \mathring M$ with $f(p)\in H$; 
the set $P$ of all such points is finite and contained in $\mathring M$. 
Choose a closed smoothly bounded simply connected domain $D\subset U$ 
such that $D$ is a neighbourhood of the arc $I$ and $f(D)\cap H=\varnothing$. 
By denting $bM$  inward along a neighbourhood of the arc $I$ we find a smoothly 
bounded compact domain $M' \subset M$ diffeotopic to $M$ and such that 
\[ 
	M=M'\cup D\quad \text{and}\quad 
	\overline{M'\setminus D} \, \cap \, \overline {D\setminus M'}=\varnothing.
\] 
Thus, $(M', D)$ is a {\em Cartan pair} (cf.\ \cite[Definition 5.7.1]{Forstneric2017E}).
Note that $P\subset M'$.

By \cite[Proposition 2.2]{AlarconForstnericLarusson2019X} there are homogeneous coordinates 
$[z_0:z_1:z_2:z_3]$ on $\CP^3$ with $H=\{z_0=0\}$ such that the contact form on 
$\CP^3\setminus H \cong \C^3= \{z_0=1\}$ is given by $dz_1+z_2dz_3-z_3dz_2$, and 
in these coordinates $f$ is of the form
\begin{equation} \label{eq:Fhg}
	f= \Fscr(g,h) = \left[1: e : g : h\right],\quad e=gh - 2\int g dh=2\int hdg-gh,
\end{equation}
where $g,h:M\to\CP^1$ are meromorphic functions on $M$ having at most simple poles at the points in $P$ 
(this reflects the fact that $f$ intersects $H$ transversely at these points and hence is an immersion 
there) and of class $\Cscr^1$ near the boundary of $M$ (in particular, $g$ and $h$ are holomorphic on 
$\mathring M\setminus P$), and $gdh$ is an exact meromorphic $1$-form with a meromorphic 
primitive $\int gdh$ determined up to an additive constant. In fact, every holomorphic Legendrian map
into $\CP^3$ intersecting the hyperplane $H$ transversely is of this form, and such an $f$ is an immersion 
if and only if the map $(g,h):M\setminus P\to\C^2$ is an immersion 
(cf.\ \cite[Corollary 2.3]{AlarconForstnericLarusson2019X}).

The meromorphic $1$-form $gdh$ on $M$ is exact if and only if $\int_C gdh=0$ for every closed curve $C$ in 
$\mathring M\setminus P$. There are two types of curves to consider: those in a homology basis of 
$H_1(M,\Z)\cong \Z^l$, say $C_1,\ldots,C_l$, and small loops around the poles of $gdh$. 
Since $M'$ is a deformation retract of $M$, the curves $C_i$ forming a homology basis of $M$
can be chosen in $\mathring M'\setminus P$ and such that they meet at a single point $p_0\in\mathring M'$ 
and their union $\bigcup_{i=1}^l C_i$ is Runge in $M$ (i.e., holomorphically convex in $M$). 
The integral of $gdh$ along a small Jordan curve around a pole $a\in P$ equals $2\pi \imath \, \Res_a (gdh)$. 
Assuming that $a$ is a simple pole of $g$ or $h$ (as is the case in our situation), 
vanishing of this integral is equivalent to
\begin{equation}\label{eq:Res}
	c_{-1}(h,a)c_1(g,a) - c_{-1}(g,a) c_{1}(h,a)=\Res_a (gdh)=0,
\end{equation}
where $c_k(h,a)$ denotes the coefficient of  $(\zeta-a)^k$ in the Laurent series 
for $h$ at $a$ in a local holomorphic coordinate $\zeta$ 
(so $c_{-1}(h,a)=\Res_a h$); see \cite[Proposition 2.4]{AlarconForstnericLarusson2019X}.   
Clearly these vanishing conditions are preserved if we replace $(g,h)$ by any pair $(g',h')$ 
of meromorphic functions which agrees with $(g,h)$ to the second order at every point $a\in P$.

Let $\Ascr^1(M;P,g)$ denote the space of meromorphic functions on $M$ which are of class 
$\Cscr^1$ up to the boundary, they have poles only at the points of $P$, and they agree with
$g$ to the second order at each point of $P$ (i.e., the difference has a second order
zero). Likewise, $\Ascr^1(M;P,h)$ denotes the corresponding space for $h$.
Consider the {\em period map} 
\[
	\Pcal=(\Pcal_1,\ldots,\Pcal_l)  : \Ascr^1(M;P,g) \times \Ascr^1(M;P,h) \to\C^l
\]
whose $j$-th component equals
\begin{equation}\label{eq:periodmap}
	\Pcal_j(x,y)=\int_{C_j} x\, dy,\qquad x\in \Ascr^1(M;P,g),\ y \in \Ascr^1(M;P,h).
\end{equation}
Note that $\Pcal(x,y)=0$ if and only if the 1-form $xdy$ is exact if and only if
the map $\Fscr(x,y):M\to \CP^3$ \eqref{eq:Fhg} is a holomorphic Legendrian curve. 
Exactness at the points of $P$ is ensured by \eqref{eq:Res} and the 
definition of the spaces $\Ascr^1(M;P,g)$ and $\Ascr^1(M;P,h)$.

The idea of proof is to first use the Riemann-Hilbert deformation method for holomorphic curves
without paying attention to the Legendrian condition. Applying this technique to the central curve 
$f:M\to\CP^3$ and the family of boundary discs $F(u,\cdotp)$ yields a new holomorphic curve 
$\tilde f:M\to\CP^3$ which satisfies the conditions in Theorem \ref{th:RH} but is not necessarily Legendrian. 
In fact, as shown in \cite[proof of Lemma 3.2]{AlarconForstnericLopez2017CM} the deviation from 
the Legendrian condition is not even pointwise small due to the fast turning of the curve 
$\tilde f$ from being close to $f$ on $M\setminus V$ (see condition (i)) to being close to the 
union of the boundary discs $F(u,\cdotp)$ when the point of $M$ approaches the boundary arc $I$
(see conditions (ii) and (iii)). However, what makes the method feasible is that the integral 
of the error is uniformly small, and hence it is possible to correct it and find nearby a Legendrian solution. 
For technical reasons which will become apparent in the proof, we shall apply the 
Riemann-Hilbert deformation method not only to a single data, but to a holomorphically
varying collection (a spray) of data of the same kind which we shall now construct.
By doing things right, the new family of curves will still satisfy the approximation conditions
for small values of the parameter, and the family will contain a Legendrian curve.
We now explain the details of this idea.

Since the map $f=[1:e:g:h]$  \eqref{eq:Fhg} is nonconstant, one of the components $g,h$ is nonconstant. 
Assume that $h$ is nonconstant; the other case can be handled by a symmetric argument.
Then, $h|_{C_j}$ is nonconstant for any $j=1,\ldots,l$ by the identity principle. Since the compact set 
$\bigcup_{j=1}^l C_j$ is Runge in $M$ and every pair of curves $C_i,C_j$ with $i\ne j$ only meet at a point, 
we can find holomorphic functions $g_1,\ldots,g_l$ on $M$ vanishing to the second order at every point of $P$ 
such that for every $j,k=1,\ldots,l$ the number $\int_{C_j} g_k\, dh\approx \delta_{j,k}$ is 
arbitrarily close to $1$ if $j=k$ and to $0$ if $j\ne k$. (Here, $\delta_{j,k}$ is the Kronecker symbol. 
To find such function, we first construct smooth functions $g_k$ on $\bigcup_{j=1}^l C_j$ such that
$\int_{C_j} g_k\, dh =\delta_{j,k}$ and then use Mergelyan's approximation theorem and  
Weierstrass's interpolation theorem to approximate them by holomorphic functions with the stated 
properties on $M$. The elementary details are left to the reader; 
see \cite[Lemma 5.1]{AlarconForstneric2014IM} or \cite[Lemma 3.2]{AlarconForstneric2019JAMS}
for the details in a similar situation when constructing minimal surfaces in $\R^n$.) 
Let $\zeta=(\zeta_1,\ldots,\zeta_l)\in\C^l$. Consider the meromorphic function 
$\hat g : M\times \C^l\to \CP^1$ given by 
\begin{equation}\label{eq:X-spray}
	\hat g(u,\zeta) = g(u) + \sum_{k=1}^l \zeta_k\, g_k(u),\quad u\in M,\ \zeta\in\C^l. 
\end{equation}
Note that $\hat g(\cdotp,\zeta)\in \Ascr^1(M;P,g)$ for every fixed $\zeta$. For all $j,k\in\{1,\ldots,l\}$ we have 
\begin{equation}\label{eq:dizetak}
	\frac{\di}{\di\zeta_k}\bigg|_{\zeta=0} \int_{C_j} \hat g(\cdotp,\zeta)\, dh 
	= \int_{C_j} g_k\, dh \, \approx\, \delta_{j,k}.
\end{equation}
If the above approximations are close enough then 
\[ 
	\frac{\di}{\di\zeta}\bigg|_{\zeta=0} \Pcal(\hat g(\cdotp,\zeta),h)  : \C^l \lra \C^l
	\ \ \text{is an isomorphism}.
\]
(A map \eqref{eq:X-spray} with this property is called a {\em period dominating holomorphic spray} 
with the core $\hat g(\cdotp,0)=g$.)  
By the inverse function theorem there is a ball $r\B\subset \C^l$ around the origin such that the period map 
$r\B\ni \zeta \mapsto \Pcal(\hat g(\cdotp,\zeta),h)\in \C^l$ is biholomorphic onto its image, the latter being a neighbourhood of the origin in $\C^l$.

Fix a point $u_0\in \mathring D$. Consider the function $\tilde e : D\times \C^l\to \C$ given by
\[
	\hat e(u,\zeta) = \hat g(u,\zeta) h(u) - 2\int_{u_0}^u \hat g(\cdotp,\zeta)\, dh + c_0,
	\quad u\in D,\ \zeta\in \C^l,
\]
where the constant $c_0\in\C$ is chosen such that $\hat e(u_0,0)=e(u_0)$,
and hence $\hat e(\cdotp,0)=e|_D$. (The integral is independent of the path in the disc $D$.
It is however impossible to extend $\hat e(\cdotp,\zeta)$ to all of $M$ since the $1$-form 
$\hat g(\cdotp,\zeta)\, dh$ has nonvanishing periods for $\zeta\ne 0$.) 
Let $\hat f : D\times \C^l\to\C^3$ be the family of holomorphic Legendrian discs
\begin{equation}\label{eq:tildef}
	D \ni u \mapsto \hat f(u,\zeta)=\left[1:\hat e(u,\zeta) : \hat g(u,\zeta):h(u)\right] \in \CP^3
\end{equation}
of the form \eqref{eq:Fhg} and depending holomorphically on $\zeta\in\C^l$. Note that $\hat f(u,0)=f(u)$ 
for $u\in D$. Since these discs lie in the affine chart $\CP^3\setminus H$, we delete the initial 
component $1$ and consider them as discs in $\C^3$. Using the same affine coordinates,  
we write the given Legendrian discs $F(u,\cdotp)$ in the theorem as 
\[
	F(u,v)= (Z(u,v),X(u,v),Y(u,v)),\quad u\in bM,\ v\in\cd.
\]
In view of \eqref{eq:Fhg} we have that 
\[
	Z(u,v)=X(u,v) Y(u,v) - 2\int_{0}^v X(u,t)dY(u,t) + e(u)-g(u)h(u).  
\]
For each point $u \in bD\cap bM$ and for every $\zeta\in\C^l$ we let
\[
	\cd  \ni v\mapsto \wh F(u,v,\zeta)= \bigl(\wh Z(u,v,\zeta),\wh X(u,v,\zeta),Y(u,v)\bigr) \in \C^{3}
\]
be the Legendrian disc of class $\Ascr^1(\cd)$ given by
\begin{eqnarray*}
	\wh X(u,v,\zeta) &=& X(u,v)+\hat g(u,\zeta)-g(u), \\
	\wh Z(u,v,\zeta) &=& \wh X(u,v,\zeta)Y(u,v) - 2\int_{t=0}^{t=v} \wh X(u,t,\zeta)\, dY(u,t) + 
	\hat e(u,\zeta) - \hat g(u,\zeta)h(u). 
\end{eqnarray*}
Note that $\wh F(u,v,0)=F(u,v)$ and
\[
	\wh F(u,0,\zeta)=\hat f(u,\zeta),\quad u\in bD\cap bM,\ \zeta\in\C^l.
\]
Finally, for every point $u\in bD\cap bM\setminus I$ and for all $\zeta\in\C^l$ we have that 
\[
	\wh F(u,v,\zeta)=\wh F(u,0,\zeta)=\hat f(u,\zeta),\quad v\in\cd,
\] 
so $\wh F(u,\cdotp,\zeta)$ is the constant disc. We extend $\wh F$ to all points $u\in bD$ by setting 
\[
	\text{$\wh F(u,v,\zeta)=\hat f(u,\zeta)$ for all $u\in bD\setminus I$, $v\in \cd$  and $\zeta\in\C^l$.}
\]
Note that for every fixed $\zeta\in\C^l$ the Legendrian disc $\hat f(\cdotp,\zeta)\colon D\to \C^3$ and 
the family of Legendrian discs $\wh F(u,\cdotp,\zeta) : \cd \to\C^3$ $(u\in bD)$ 
satisfy the assumptions of \cite[Lemma 3.2]{AlarconForstnericLopez2017CM} on $D$ (which is conformally 
diffeomorphic to the standard disc $\cd$), and both families depend holomorphically on $\zeta\in\C^l$. 
Hence, \cite[Lemma 3.2]{AlarconForstnericLopez2017CM} furnishes a family of Legendrian discs 
\[
	\wh G(\cdotp,\zeta) =(\wh G_1,\wh G_2,\wh G_3): D\to\C^3
\]
depending holomorphically on $\zeta$ and satisfying
the estimates in the lemma uniformly with respect to $\zeta \in r\B$. (These estimates are of 
the same type as those in conditions (i)--(iii) of Theorem \ref{th:RH} with $M$ replaced by $D$.
The observation regarding holomorphic dependence and uniformity 
of the estimates with respect to $\zeta$ is evident from 
\cite[proof of Lemma 3.2]{AlarconForstnericLopez2017CM} and has also been used in 
\cite[proof of Theorem 3.3]{AlarconForstnericLopez2017CM}.)

Let $V\subset D\setminus M'$ be a small neighbourhood of the arc $I\subset bM$.
By \cite[Lemma 3.2 (iv)]{AlarconForstnericLopez2017CM} we may assume that 
$\wh G(\cdotp,\zeta)$ is as close as desired to $\hat f(\cdotp,\zeta)$ in the $\Cscr^1$ norm on 
$D\setminus V$, and hence on $M'\cap D\subset D\setminus V$ for all $\zeta\in r\B$. 
In particular, given $\delta >0$ we may assume that
\[
	\|\wh G(\cdotp,\zeta)-\hat f(\cdotp,\zeta)\|_{\Cscr^1(M'\cap D)} <\delta, \quad \zeta\in  r\B.
\]
Recall that the component $\hat g$ of $\hat f$ \eqref{eq:tildef} is a meromorphic function
on $M\times \C^l$ with poles only on $P\times \C^l$.  
By solving a Cousin-I problem with bounds on the Cartan pair $(M',D)$, with interpolation on 
$P$, we can glue $\hat g$ and $\wh G_2$ 
into a function $H_2(\cdotp,\zeta): M\to \CP^1$ of class $\Ascr^1(M;P,g)$, 
holomorphic in $\zeta\in r\B$ and satisfying the estimates
\[
	\|H_2(\cdotp,\zeta)-\hat g(\cdotp,\zeta)\|_{\Cscr^1(M')} < C \delta,\quad\   
	\|H_2(\cdotp,\zeta)-\wh G_2(\cdotp,\zeta)\|_{\Cscr^1(D)} < C \delta,
\]
where the constant $C>0$ only depends on the Cartan pair $(M',D)$. 
By the same token, we can glue the last component $h$ of $\hat f$ 
with the function $\wh G_3(\cdotp,\zeta)$ into a function $H_3(\cdotp,\zeta)$ of class 
$\Ascr^1(M;P,h)$, depending holomorphically on $\zeta\in  r\B$ and satisfying the estimates
\[
	\|H_3(\cdotp,\zeta)- h\|_{\Cscr^1(M')} < C  \delta,\quad  \ 
	\|H_3(\cdotp,\zeta)-\wh G_2(\cdotp,\zeta)\|_{\Cscr^1(D)} < C  \delta.
\]
Since $\bigcup_{i=1}^l C_i\subset M'\setminus P$, 
it follows that the period map $\zeta \mapsto \Pcal(H_2(\cdotp,\zeta), H_3(\cdotp,\zeta))$ 
(see \eqref{eq:periodmap}) approximates the biholomorphic map 
$\zeta\mapsto \Pcal(\hat g(\cdotp,\zeta), h)$ uniformly on $\zeta\in  r\B$.  
Assuming that $\delta>0$ is chosen small enough, there is a point $\zeta'\in  r\B$ as close 
to the origin as desired such that 
\[ 
	\Pcal\bigl(H_2(\cdotp,\zeta'), H_3(\cdotp,\zeta')\bigr)=0.
\] 
Setting $\tilde g=H_2(\cdotp,\zeta')$, $\tilde h=H_3(\cdotp,\zeta')$ we obtain a holomorphic 
Legendrian curve 
\[
	\tilde f =[1:\tilde e:\tilde g:\tilde h] : M\to \CP^3
\]
of the form \eqref{eq:Fhg} with $\tilde e(u_0)=e(u_0)$. It follows from the construction that 
$\tilde f$ satisfies conditions (i)--(iii) of Theorem \ref{th:RH} provided 
that the approximations were close enough.

In order to ensure also the interpolation condition (iv) at finitely many points 
$A=\{a_1,\ldots, a_k\}\subset \mathring M$, we amend the above procedure as follows. 
First, we choose the hyperplane $H$ at the beginning of the proof such that, in addition to the other 
stated conditions, it does not intersect the finite set $f(A)$, and we choose the disc $D$ as above and 
contained in $M\setminus A$. Pick a base point $u_0\in \mathring D$. We connect $u_0$ 
to each point $a_j\in A$ by an embedded oriented arc $E_j\subset \mathring M\setminus P$ 
which exits $D$ only once and such that any two of these arcs only meet at $u_0$.
It follows that the inclusion $M\setminus (D\cup \bigcup_{i=1}^k E_i) \subset M$ is a homotopy 
equivalence, and hence we can choose curves $C_1,\ldots,C_l$ forming a homology basis of $M$ 
contained in the complement of $D\cup P\cup \bigcup_{i=1}^k E_i$.
To the period map $\Pcal$ \eqref{eq:periodmap} we add $k$ additional
components given by the integrals over the arcs $E_1,\ldots,E_k$.
The rest of the proof remains unchanged. By ensuring that 
the integrals of the $1$-form $\tilde g d\tilde h$ over the arcs $E_1,\ldots,E_k$ equal those of $gdh$, 
the map $\tilde f$ agrees with $f$ at the points of $A$. By the same tools we can 
obtain finite order interpolation on $A$.
\end{proof}

%
%

\section{A general position theorem for Legendrian curves in projective spaces}\label{sec:GP}

Holomorphic Legendrian curves obtained by Riemann-Hilbert modifications  in the previous
section typically have branch points. However, in the application of this method to the proof
of Theorem \ref{th:main} we need Legendrian immersions.

The purpose of this section is to explain the following general position theorem for holomorphic 
Legendrian curves in projective spaces. As was already pointed out in the previous section,
$\CP^{2n+1}$ admits a unique complex contact structure up to holomorphic contactomorphisms 
a hence a concrete choice of the contact bundle is irrelevant. 

%
%
\begin{theorem}\label{th:emb}
\begin{enumerate}[\rm (a)]
\item 
Let $M$ be a compact bordered Riemann surface. Every Legendrian curve $f:M\to\CP^{2n+1}$ of class 
$\Ascr^1(M)$ can be approximated in the $\Cscr^1(M,\CP^{2n+1})$ topology 
by holomorphic Legendrian embeddings $\tilde f: M\hra \CP^{2n+1}$.
\item
Every holomorphic Legendrian curve $f:M\to\CP^{2n+1}$ from an open Riemann surface can be approximated
uniformly on compacts  in $M$ by holomorphic Legendrian embeddings $\tilde f:M\hra \CP^{2n+1}$.
\end{enumerate}
\end{theorem}

The analogue of this result for Legendrian curves in $\C^{2n+1}$ with its standard complex contact structure
was proved in \cite[Theorem 5.1]{AlarconForstnericLopez2017CM} where it was shown in addition 
that the approximating embedding $\tilde f:M\hra \C^{2n+1}$ in case (b) can be chosen proper.
(The latter condition is of course impossible in the compact manifold $\CP^{2n+1}$.) 
The cited result also gives approximation of generalised Legendrian curves $S\to \C^{2n+1}$  
on compact Runge admissible sets $S$ in an open Riemann surface $M$; this can be 
extended to Legendrian curves in $\CP^{2n+1}$ as well, but we shall not need it in the present paper.

\begin{proof}
It was shown in \cite[Corollary 3.7]{AlarconForstnericLarusson2019X} that every holomorphic 
Legendrian immersion $M\to\CP^{2n+1}$ from an open Riemann surface $M$ 
can be approximated uniformly on compacts by holomorphic Legendrian embeddings $M\hra\CP^{2n+1}$. 
The proof combines \cite[Theorem 1.2]{AlarconForstneric2019IMRN} 
to the effect that every holomorphic Legendrian immersion $M\to X$ from an open Riemann surface
to an arbitrary complex contact manifold can be approximated, uniformly on any compact 
subset $K$ of $M$, by holomorphic Legendrian embeddings $U\hra X$ from open neighbourhoods
$U\subset M$ of $K$, and the approximation theorem for holomorphic Legendrian immersions
into projective spaces given by \cite[Theorem 3.4]{AlarconForstnericLarusson2019X}.

To prove the theorem, it remains to show that 
one can approximate any Legendrian map $f:M\to \CP^{2n+1}$ of class $\Ascr^1(M)$ 
by holomorphic Legendrian immersions $U\to \CP^{2n+1}$ from open neighbourhoods $U$ of $M$ 
in an ambient Riemann surface. For the convenience of notation we consider curves in $\CP^3$,
although this restriction is inessential. As in the proof of Theorem \ref{th:RH} we find a projective hyperplane 
$H\subset \CP^3$ intersecting $f(M)$ transversely in at most finitely many points $P\subset \mathring M$ 
and not intersecting  $f(bM)$, and homogeneous coordinates $[z_0:z_1:z_2:z_3]$ 
with $H=\{z_0=0\}$ in which 
$
	f= \Fscr(g,h) = \left[1: gh - 2\int g dh : g : h\right]
$
(cf.\ \eqref{eq:Fhg}), where $g,h:M\to\CP^1$ are meromorphic functions having only simple poles 
at the points in $P$ and of class $\Cscr^1$ near the boundary of $M$.
A map $f$ of this form is an immersion if and only if  $(g,h):M\setminus P\to\C^2$ is an immersion 
(cf.\ \cite[Corollary 2.3]{AlarconForstnericLarusson2019X}). It now suffices to approximate the map 
$(g,h):M\to (\CP^1)^2$ as closely as desired in $\Cscr^1(M,(\CP^1)^2)$ by a meromorphic map 
$(\tilde g,\tilde h):U\to (\CP^1)^2$ defined on a neighbourhood $U\subset R$ of $M$ such that 
$(\tilde g,\tilde h)$ agrees with $(g,h)$ to the second order at every point of $P$,
it is a holomorphic immersion $U\setminus P\to \C^2$, and the meromorphic $1$-form 
$\tilde g d\tilde h$ is exact. We have seen in the proof of Theorem \ref{th:RH}
that the interpolation condition on $P$ ensures that $\tilde g d\tilde h$ 
has a local meromorphic primitive at every point of $P$; see \eqref{eq:Res}.
Therefore, exactness of $\tilde g d\tilde h$ is equivalent to the period 
vanishing conditions $\int_{C_i} \tilde g d\tilde h=0$ $(i=1,\ldots,l)$ where 
$C_1,\ldots, C_l\subset \mathring M\setminus P$ is a basis of the homology group $H_1(M,\Z)=\Z^l$. 

The construction of $(\tilde g,\tilde h)$ satisfying these conditions is made in two steps.
In the first step we approximate $(g,h)$ by a meromorphic map $(\hat g,\hat h)$ defined on a 
neighbourhood $U\subset R$ of $M$ 
which agrees with $(g,h)$ to the second order at the points of $P$, it has no poles
on $M\setminus P$, and such that $\hat g d\hat h$ is exact. This is achieved by following the proof of 
\cite[Lemma 4.3]{AlarconForstnericLopez2017CM}, the only addition being the presence
of poles at the points in $P$ and the interpolation condition at these points. 
Next, we follow the first part of the proof of 
\cite[Lemma 4.4]{AlarconForstnericLopez2017CM} in order to approximate 
$(\hat g,\hat h)$ on $M$ and interpolate it to the second order on $P$ by a 
meromorphic map $(\tilde g,\tilde h)$ on a neighbourhood of $M$ which is a holomorphic 
immersion of $M\setminus P$ into $\C^2$ and such that $\tilde gd\tilde h$ is an exact
meromorphic $1$-form. By what has been said,
the associated map $\tilde f=\Fscr(\tilde g,\tilde h):M\to\CP^3$ 
defined by \eqref{eq:Fhg} is then a holomorphic immersion. Both proofs alluded to above
easily extend to the present setting in essentially the same way as was done in the proof
of Theorem \ref{th:RH}, and we leave the details to the reader.
\end{proof}

\begin{problem}
Does part (a) of Theorem \ref{th:emb} hold for Legendrian curves in an arbitrary complex contact manifold $(X,\xi)$?
\end{problem}

Assuming that $f:M\to X$ is a Legendrian {\em immersion} of class $\Ascr^2(M,X)$, 
it was shown in \cite[Theorem 1.2]{Forstneric2020Merg} that $f$ can be
approximated in $\Cscr^2(M,X)$ by holomorphic Legendrian embeddings of small open neighbourhoods 
of $M$ into $X$; however, the cited result does not apply to branched Legendrian maps.

%
%

\section{Proof of Theorem \ref{th:main}}\label{sec:proof}

Let $\Omega\subset \CP^3$ be the domain \eqref{eq:Omega2} and $\pi:\Omega\to \B\subset \R^4$
be the twistor bundle over the hyperbolic ball $(\B,g_h)$ given by \eqref{eq:pi}.
Denote by $\xi\subset T\CP^3$ the holomorphic contact bundle determined by the 
homogeneous $1$-form $\beta$ \eqref{eq:beta}, so $\xi|_{T\Omega}$ is the horizontal bundle of the twistor
projection $\pi:\Omega\to \B$. When speaking of Legendrian curves in $\Omega$, we always mean
holomorphic curves tangent to $\xi$. By what has been said in Sect.\ \ref{sec:H4}, 
Theorem \ref{th:main} follows immediately from the following result.

%
%
\begin{theorem}\label{th:Omega}
Let $M$ be a bordered Riemann surface and $F:\overline M\to \Omega$ be a Legendrian curve of class 
$\Cscr^1(\overline M,\Omega)$ which is holomorphic on $M$. Then, $F$ can be approximated uniformly 
on compacts in $M$ by proper holomorphic Legendrian embeddings $\wt F:M\hra \Omega$ which
can be chosen to agree with $F$ to a given finite order at  finitely many points in $M$.
\end{theorem}

Indeed, by the Bryant correspondence the given superminimal immersion $f:\overline M\to \B$ 
in Theorem \ref{th:main} (which may be assumed of positive spin) lifts to a unique 
Legendrian immersion $F:\overline M\to\Omega$ as in Theorem \ref{th:Omega}, and if
$\wt F:M\hra \Omega$ is a resulting proper holomorphic Legendrian embedding in Theorem \ref{th:Omega} 
then its projection $\tilde f=\pi\circ \wt F:M\to \B$ is a proper superminimal immersion satisfying
Theorem \ref{th:main}.

We begin with some preparations. Consider the exhaustion function $\rho:\Omega\to [0,1)$ defined 
in the homogeneous coordinates $z=[z_1:z_2:z_3:z_4]$ by 
\begin{equation}\label{eq:rho}
	\rho([z_1:z_2:z_3:z_4])=|\pi(z)|^2 = \frac{|z_3|^2+|z_4|^2}{|z_1|^2+|z_2|^2}
\end{equation}
(see \eqref{eq:pi}).  Given a pair of numbers $0<c<c'\le 1$ we write
\begin{equation}\label{eq:Omegac}
	 \Omega_{c}=\{z\in\Omega: \rho(z)<c\}, \ \quad 
	 \Omega_{c,c'}=\{z\in\Omega: c<\rho(z)<c'\}.
\end{equation}

For every point $z\in \Omega\setminus \pi^{-1}(0)$ there is a unique properly
embedded Legendrian disc $L_z\subset\Omega$ with $z\in L_z$ whose projection 
$\pi(L_z)\subset \B$ is a hyperbolic surface $\Lambda(\pi(z),V)$ in Proposition \ref{prop:Lambda}.
Indeed, by the twistor correspondence the point $z$ represents an almost hermitian
structure on the tangent space $T_x\R^4$ at the base point $x=\pi(z)\in\B\setminus \{0\}$. 
Let $S_x\subset\B$ denote the three-sphere with centre $0$ and passing through $x$. Then, 
\begin{equation}\label{eq:Lz}
	\pi(L_z) = \Lambda(x,V)\ \ \text{where}\ \ V=T_x S_x \cap z(T_x S_x).
\end{equation}
That is, $V$ is the unique $z$-invariant two-plane contained in the three dimensional
tangent space $T_x S_x$ to the sphere $S_x$ at $x$. (Such $L_z$ also exists for every point $z$ 
in the central fibre $\pi^{-1}(0)$, but it is not unique since different $2$-planes $V\subset T_0\R^4$ 
may determine the same almost hermitian structure $z$ on $T_0\R^4$.) 
By Proposition \ref{prop:Lz}, $L_z$ is the intersection of $\Omega$ with a linearly embedded 
Legendrian rational curve $\CP^1\subset \CP^3$. By Proposition \ref{prop:Lambda} we have 
\begin{equation}\label{eq:Lz2}
	L_z \subset \{z\} \cup \Omega_{c,1}\ \ \text{where}\ \ c=|\rho(z)|\in (0,1), 
\end{equation}
where we are using the notation \eqref{eq:Omegac}. 
It is obvious that the family of Legendrian holomorphic discs $L_z$ depend real-analytically on 
the point $z\in \Omega\setminus \pi^{-1}(0)$.

Theorem \ref{th:Omega} is obtained from the following lemma by a standard inductive procedure.

%
%
%
%
\begin{lemma}\label{lem:bigstep}
Let $M$ be a bordered Riemann surface, $P$ be a finite set of points in
$M$, and $0<c<c'<c''<1$. Assume that $F:\overline M\to \Omega$ is a 
Legendrian map of class $\Ascr^1(\overline M,\Omega)$ and $U\Subset M$ is an open 
subset such that $F(M\setminus U)\subset\Omega_{c,c'}$. Given $\epsilon>0$ there exists a 
holomorphic Legendrian embedding $G: \overline M \to\Omega$ satisfying the following conditions:
\begin{enumerate}[\rm (i)]
\item  $G(bM)\subset  \Omega_{c',c''}$,
\item  $G(M\setminus U)\subset \Omega_{c,c''}$,
\item  $\dist(G(u),F(u))<\epsilon$ for $u\in \overline U$, and
\item  $F$ and $G$ have the same $k$-jets at each of the points in $P$ for a given $k\in\N$. 
\end{enumerate}
\end{lemma}

The details of proof that Lemma \ref{lem:bigstep} implies Theorem \ref{th:Omega} are left to the reader. 
Inductive constructions of this type are ubiquitous in the literature; see e.g.\ 
\cite[proof of Theorem 1.1]{DrinovecForstneric2007DMJ} 
using \cite[Lemma 6.3]{DrinovecForstneric2007DMJ}
and note that our situation is simpler since the exhaustion function $\rho$ \eqref{eq:rho}
of $\Omega$ has no critical points in $\Omega\setminus\pi^{-1}(0)$. 
In order to ensure that the limit map of this procedure is a Legendrian embedding, we use the general 
position theorem (see Theorem \ref{th:emb}) at each step and approximate sufficiently 
closely in subsequent steps.

\begin{proof}[Proof of Lemma \ref{lem:bigstep}]
Given a pair of numbers $0<c<c'<1$ and an open set $\omega\subset b\Omega_c=\rho^{-1}(c)$
(see \eqref{eq:Omegac}), we let
\begin{equation}\label{eq:DUc}
	D(\omega,c,c') := \Omega_{c'} \setminus \bigcup_{z\in b\Omega_c\setminus \omega}L_z.
\end{equation}
Clearly, $D(\omega,c,c')$ is an open set containing $\Omega_c$ and we have that 
\begin{equation}\label{eq:avoiding}
	z\in \Omega\setminus D(\omega,c,c') \ \Longrightarrow \ 
	L_z\subset  \Omega\setminus D(\omega,c,c').
\end{equation}
Moreover, it is elementary to see that 
there is a subdivision $c=c_0<c_1<\cdots < c_m=c'$ and for every $i=0,1,\ldots,m-1$ 
a finite open cover $\omega_{i,1},\ldots,\omega_{i,k_i}$ of $b\Omega_{c_i}$ such that 
\begin{equation}\label{eq:exhausting}
	\bigcup_{j=1}^{k_i} D(\omega_{i,j},c_{i},c_{i+1}) = \Omega_{c_{i+1}},
\end{equation}
and for every $i$ as above and $j=1,\ldots,k_i$ we also have that
\begin{equation}\label{eq:affine}
	\bigcup_{z\in \overline{D(\omega_{i,j},c,c')\setminus \Omega_{c_i}}} 
	L_z\cap \overline \Omega_{c_{i+1}} \ \ \text{is contained in an affine chart of $\CP^3$}.
\end{equation}

We are now ready to prove the lemma. Let us begin by explaining the initial step.
The assumptions imply that $F(bM) \subset \Omega_{c,c'}$. Consider the set 
\begin{equation}\label{eq:Iprime1}
	I'_1=\{u\in bM: F(u)\in D_1:=D(\omega_{1,1},c_0,c_1)\}. 
\end{equation}
Assume first that $I'_1$ does not contain any boundary component of $M$.
Then, $I'_1$ is contained in the interior of the union $I=\bigcup_{i=1}^j I_i$ of finitely many pairwise 
disjoint segments $I_1,\ldots,I_j\subset bM$ none of which is a component of $bM$.
Choose a number $c'_1$ with $c_1<c'_1<c_2$ and close to $c_1$. Consider the Riemann-Hilbert problem 
(cf.\ Theorem \ref{th:RH}) with the central Legendrian curve $F:M\to\Omega$ and the family 
of Legendrian discs $\wh L_u:=L_{F(u)}\cap \overline \Omega_{c'_1}$ for points $u\in I$. 
(In Theorem \ref{th:RH}, the central disc is denoted $f$ and parameterizations of the boundary 
discs are denoted $F(u,\cdotp)$.) For the values $u\in I\setminus I'_1$ we shrink the discs $\wh L_u$
within themselves (by dilations) to reach the constant discs $\wh L_u=\{F(u)\}$ as $u$ reaches the 
boundary of $I$; these discs remain in the complement of $\overline D_1$ in view of \eqref{eq:avoiding}.  
By \eqref{eq:affine} and decreasing $c'_1>c_1$ if necessary we can also arrange that the set
$\bigcup_{u\in I} \wh L_u$ is contained in an affine chart of $\CP^3$.
Applying Theorem \ref{th:RH} to this configuration gives a new holomorphic Legendrian curve 
$F':\overline M\to \Omega$ whose boundary $F'(bM)\subset \Omega_{c,c'}$ no longer intersects 
$\overline D_1$ and the remaining conditions in the theorem are satisfied. If however the set $I'_1$ 
\eqref{eq:Iprime1} contain a boundary component of $M$, we perform the 
same procedure twice, first pushing a part of $I'_1$ out of $\overline D_1$
and thereby reducing to the previous case. 

The subsequent steps are basically the same as the first one. For simplicity 
we denote the result of step 1 again by $F$, so in step 2 the assumption is that 
$F(bM)\subset \Omega_{c'} \setminus \overline D_1$. By following the same procedure we 
push the boundary of $M$ out of the set $\overline{D_1\cup D_2}$ where
$D_2:=D(\omega_{1,2},c_0,c_1)$. Note that a point of $F(bM)$ which is outside of $\overline D_1$ 
will not reenter this set in subsequent steps in view of condition \eqref{eq:avoiding}. 
We see from \eqref{eq:exhausting} that in $k_1$ steps of this kind the image of $bM$ is pushed 
into $\Omega_{c_1,c'}$. We then continue inductively to the next levels $c_2,\ldots,c_m=c'$, 
eventually pushing the image of $bM$ into the domain $\Omega_{c',c''}$ by a Legendrian map
$G$ satisfying the conditions in the lemma.
\end{proof}

%
%
\subsection*{Acknowledgements}
Research on this paper was supported by the program P1-0291 and grant J1-9104
from ARRS, Republic of Slovenia, and by the Stefan Bergman Prize 2019 awarded by the 
American Mathematical Society.




\vspace*{5mm}
\noindent Franc Forstneri\v c

\noindent Faculty of Mathematics and Physics, University of Ljubljana, Jadranska 19, SI--1000 Ljubljana, Slovenia, and 

\noindent 
Institute of Mathematics, Physics and Mechanics, Jadranska 19, SI--1000 Ljubljana, Slovenia

\noindent e-mail: {\tt franc.forstneric@fmf.uni-lj.si}

\end{document}